\documentclass[reqno]{amsart}
\usepackage{hyperref}
\usepackage{amssymb}

\begin{document}
\title[One side invertibility ]
{One side invertibility for implicit hyperbolic systems with delays}

\author[F. Haddouchi]
{Faouzi Haddouchi}

\address{Faouzi Haddouchi \newline
Department of Physics\\
University of Sciences
and Technology\\
USTO, El M'naouar, BP 1505, Oran, 31000, Algeria}
\email{haddouch@univ-usto.dz}


\begin{abstract}
This paper deals with left invertibility problem of implicit
hyperbolic systems with delays in infinite dimensional Hilbert
spaces. From a decomposition procedure, invertibility for this class
of systems is shown to be equivalent to the left invertibility of a
subsystem without delays.
\end{abstract}

\maketitle \numberwithin{equation}{section}
\newtheorem{theorem}{Theorem}[section]
\newtheorem{corollary}[theorem]{Corollary}
\newtheorem{proposition}[theorem]{Proposition}
\newtheorem{remark}[theorem]{Remark}
\newtheorem{definition}[theorem]{Definition}
\newtheorem{lemma}[theorem]{Lemma}
\allowdisplaybreaks

\newcommand{\ip}[2]{\langle #1,#2\rangle}

\section{Introduction}

We shall deal here with left invertibility for a class of implicit
systems with delays in Hilbert spaces which are either left or right
invertible (hereafter called \textquotedblleft one side
invertible\textquotedblright) and whose state is described by
\begin{equation}  \label{eq-1}
E\overset{..}{z}(t) +\alpha \overset{.}{z}
(t)+A_{0}z(t)+A_{1}z(t-h)=Bu(t)
\end{equation}
With the output function
\begin{equation}  \label{eq-2}
y(t)=Cz(t)
\end{equation}
where $z\left( 0\right) =0$, $\overset{.}{z}(0)=0$ et $z(t)\equiv 0$, $%
\forall t\in \left[ -h,0\right[ $, ($h>0$), $\alpha \geq 0$,
$\overset{..}{z}(t) $, $\overset{.}{z}(t)$, $z(t)\in H$ (the state
space); $ u(t) \in U$ (the input space); $y(t)\in Y$ (the output
space). $E$, $A_{0}$, $A_{1}$, $B$ and $C$ are linear operators
defined as follows:
$E:H\longrightarrow H\ $; $A_{0}:H\longrightarrow H$; $%
A_{1}:H \longrightarrow H$; $B:U\longrightarrow H$;
$C:H\longrightarrow Y$, where $U$, $H$ and $Y$ are real Hilbert
spaces.

The concept of left invertibility is the problem of determining the
conditions under which a zero output corresponding to a zero initial
state can only be generated by a zero input. The left invertibility
problem has been investigated within its different versions in
detail in the finite dimensional case. We cite \cite{Bon},
\cite{Min} and \cite{Silv} whose approaches are generalized in this
study. In infinite dimension, this notion has been generalized
essentially for non implicit systems with delays and for the
particular case of bounded operators (see for instance
\cite{Malab}).

The aim of this paper is to extend this notion in the direction of
infinite dimensional linear systems with delays and to use this
approach for solving left invertibility for implicit hyperbolic
systems with delays. The motivation for considering this class is
given by the article of Bonilla \cite{Bon} that gave a natural left
or right inverse for implicit descriptions without delays.

Our first contribution is to give necessary and sufficient
conditions for the system to be left invertible. The paper is
structured as follows. The first part deals with solvability and
invertibility of this class of systems in the frequency domain. The
second part deals with systems without delays. A criterion will be
given for one side invertibility in this section.

\section{Solvability and one side invertibility}

Taking $\ y=\overset{.}{z}$, $\overset{.}{y}=\overset{..}{z}$, $%
w=\left(
\begin{array}{l}
z \\
\multicolumn{1}{c}{\overset{.}{z}}%
\end{array}%
\right) $, $\ \widetilde{B}=%
\begin{pmatrix}
0 \\
B%
\end{pmatrix}%
$, \  $\widetilde{A}_{0}=%
\begin{pmatrix}
0 & I \\
-A_{0} & -\alpha I%
\end{pmatrix}%
$,\\
$\ \widetilde{A}_{1}=%
\begin{pmatrix}
0 & 0 \\
-A_{1} & 0%
\end{pmatrix}%
$,\ $\widetilde{E}=%
\begin{pmatrix}
I & 0 \\
0 & E%
\end{pmatrix}%
$, $\widetilde{C}=%
\begin{pmatrix}
0 & 0 \\
C & 0%
\end{pmatrix}%
$.
\\

Then the system \eqref{eq-1}, \eqref{eq-2} can be written as:
\begin{equation}
(\widetilde{\Sigma })\left\{
\begin{array}{l}
\widetilde{E}\overset{.}{w}(t)=\widetilde{A}_{0}w(t)+\widetilde{A}
_{1}w(t-h)+\widetilde{B}u(t) \\
\widetilde{y}(t)=\widetilde{C}w(t)%
\end{array}
\right.
\end{equation}
As usual, $u$, $w$, $\widetilde{y}$ represent respectively the
input, state and output of the system $ (\widetilde{\Sigma })$. The
setting is very general in the sense that $\widetilde{E}$ is not
invertible, $\widetilde{A}_{0}$, $\widetilde{A}_{1}$ are unbounded
operators, $\widetilde{B}$ and $\widetilde{C}$ are restricted to be
bounded, uniqueness of the solution is not required, and an explicit
solution, will even not be demanded.

By using Laplace transform (this was used in \cite{Lewi}) the system
$ (\widetilde{\Sigma })$ can be rewritten as follows:
\begin{equation}
(\widehat{\Sigma })\left\{
\begin{array}{l}
(s\widetilde{E}-\widetilde{A}_{0}-\widetilde{A}_{1}\exp (-sh))\text{
}
\widehat{w}(s)=\widetilde{B}\widehat{u}(s) \\
\widehat{y}(s) =\widetilde{C}\widehat{w}(s)%
\end{array}
\right.
\end{equation}
where $s$ is the classical Laplace variable.

\begin{definition}  \label{def 2.1}
The system $ (\widehat{\Sigma })$ is said to be solvable if
$(s\widetilde{E}- \widetilde{A}_{0}-\widetilde{A}_{1}\exp (-sh))$ is
invertible.
\end{definition}

\begin{definition} \label{def 2.2}
The system $(\widetilde{\Sigma })$ is left invertible if the
following condition is fulfilled
\begin{equation}
\widetilde{y}(t)\equiv 0\ \ \Longrightarrow u(t)\equiv 0
\end{equation}
\end{definition}
By $U_{-1}(s)$ we shall denote all functions that are strictly
proper (see \cite{Zwart} for more detail and information). This
notion may also be expressed in terms of transfer function of the
system $(\widetilde{\Sigma })$ as follows.

\begin{lemma} \label{lem 2.3}
The system $(\widetilde{\Sigma })$ is left invertible if and only if
\begin{equation}
T(s,\exp (-sh))\widehat{u}(s)\equiv 0 \ \ \Longrightarrow
\widehat{u}(s) \equiv 0
\end{equation}
\end{lemma}

where $\widehat{u}(.)\in U_{-1}(s)$ and $T(s,\exp
(-sh))=\widetilde{C}(s
\widetilde{E}-\widetilde{A}_{0}-\widetilde{A}_{1}\exp (-sh))^{^{-1}}
\widetilde{B}{ }$ is the transfer function of the system
$(\widetilde{\Sigma })$.

\begin{proof}
It directly obtained by using the Laplace transform of input-output
relation.
\end{proof}

\begin{proposition} \label{prop 2.4}
Under the condition that
\begin{equation}
\forall s \in \mathbb{C}
 \ \ker
\begin{bmatrix}
 s\widetilde{E}-\widetilde{A}_{0}-\widetilde{A}_{1}\exp (-sh)& \widetilde{B} \\
\widetilde{C}& 0
\end{bmatrix}
=\left\{ 0\right\}
\end{equation} The system $(\widetilde{\Sigma })$ is left invertible.
\end{proposition}

\begin{proof}
This can be easily proved using the Lemma \ref{lem 2.3}.
\end{proof}

\begin{proposition}  \label{prop 2.5}
 If the system $(\widetilde{\Sigma })$ is left invertible, then
\begin{equation}
ker
\begin{bmatrix}
 s\widetilde{E}-\widetilde{A}_{0}-\widetilde{A}_{1}\exp (-sh)& \widetilde{B} \\
\widetilde{C}& 0
\end{bmatrix}
=\left\{ 0\right\}
\end{equation} for some  $s \in \mathbb{C}$.
\end{proposition}

\begin{proof}
It follows from Lemma \ref{lem 2.3}.
\end{proof}
Note that, in order to simplify the exposition, we just consider
here systems having only one delay in the state. Our results may
easily be extended to systems with several integer delays in the
state, in the input and in the output.

\section{Classical systems without delays and one side invertibility}
We can associate with the system $(\widetilde{\Sigma })$ the
following quadruples of operators
$(\widetilde{D}_{k},\widetilde{F}_{k},\widetilde{G}_{k},\widetilde{H}_{k})$
representing the family of classical (without delays) implicit
systems (see for instance \cite{Olb}, \cite{Tsoi}), say
$(\widetilde{\Sigma }_{k})$:
\begin{equation}
(\widetilde{\Sigma }_{k})\left\{
\begin{array}{l}
\widetilde{D}_{k}\overset{.}{x_{k}}(t)=\widetilde{F}_{k}x_{k}(t)+\widetilde{G%
}_{k}v_{k}\left( t\right) \\
y_{k}\left( t\right) =\widetilde{H}_{k}\text{ }x_{k}(t)%
\end{array}
\right.
\end{equation}
where

\begin{center}
$\widetilde{F}_{k}=%
\begin{pmatrix}
\widetilde{A}_{0} & 0 & . & 0 \\
\widetilde{A}_{1} & . & . & . \\
. & . & . & 0 \\
0 & . & \widetilde{A}_{1} & \widetilde{A}_{0}%
\end{pmatrix}%
$,{\tiny \ \ }$\widetilde{D}_{k}=%
\begin{pmatrix}
\widetilde{E} & 0 & . & 0 \\
0 & \widetilde{E} & . & . \\
. & . & . & 0 \\
0 & . & . & \widetilde{E}%
\end{pmatrix}%
$,

$\widetilde{H}_{k}\ =%
\begin{pmatrix}
\widetilde{C} & 0 & . & 0 \\
0 & \widetilde{C} & . & . \\
. & . & . & 0 \\
0 & . & . & \widetilde{C}%
\end{pmatrix}%
$,{\tiny \ \ }$\widetilde{G}_{k}=%
\begin{pmatrix}
\text{$\widetilde{B}$} & 0 & . & 0 \\
0 & \text{$\widetilde{B}$} & . & . \\
. & . & . & 0 \\
0 & . & . & \text{$\widetilde{B}$}%
\end{pmatrix}%
$,

$x_{k}(t)=%
\begin{pmatrix}
w_{0}(t) \\
. \\
. \\
w_{k}(t)%
\end{pmatrix}%
${\tiny \ },$\ v_{k}\left( t\right) =%
\begin{pmatrix}
u_{0}(t) \\
. \\
. \\
u_{k}(t)%
\end{pmatrix}%
$,
\end{center}

$w_{k}(t)=w(t+kh)$, $u_{k}(t)=u(t+kh)$ for all $t\in \left[ 0,h\right] $, $%
w_{k}(0)=w_{k-1}(h)$, \ and $u_{k}(0)=u_{k-1}(h)$.

Let us denote the transfer function of $(\widetilde{\Sigma }_{k})$
as :
\begin{equation}
\Phi _{k}\left( s\right) =\widetilde{H}_{k}(s\widetilde{D}_{k}-\widetilde{F}%
_{k})^{^{-1}}\widetilde{G}_{k}
\end{equation}
First we shall start by giving a result for invertibility of the
operator $(s\widetilde{D}_{k}-\widetilde{F}_{k})$ .

\begin{lemma}  \label{lem 3.1}
The operator $(s\widetilde{D}_{k}-\widetilde{F}_{k})$ is invertible
if and only if $(s\widetilde{E}-\widetilde{A}_{0})$ is invertible.
\end{lemma}
\begin{proof}
i) Suppose that $(s\widetilde{E}-\widetilde{A}_{0})$ is invertible.
Let $\ v=\left( {\small x}_{0}....{\small x}_{k}\right) ^{T}$ be a
vector in the Kernel of $(s\widetilde{D}_{k}-\widetilde{F}_{k})$.
This implies

\begin{center}
$\left\{
\begin{array}{c}
(s\widetilde{E}-\widetilde{A}_{0})x_{0}=0 \\
(s\widetilde{E}-\widetilde{A}_{0})x_{1}-\widetilde{A}_{1}x_{0}=0 \\
(s\widetilde{E}-\widetilde{A}_{0})x_{2}-\widetilde{A}_{1}x_{1}=0 \\
. \\
. \\
(s\widetilde{E}-\widetilde{A}_{0})x_{k}-\widetilde{A}_{1}x_{k-1}=0%
\end{array}%
\right. $
\end{center}This is equivalent to saying that $%
x_{0}=x_{1}=....=x_{k}=0$.

ii) Reciprocally, suppose that
$(s\widetilde{D}_{k}-\widetilde{F}_{k})$ is invertible and let $\
x_{0}$ \ be an element such \ that $\
(s\widetilde{E}-\widetilde{A}_{0})x_{0}=0 $, then we have

\begin{center}
$%
\begin{pmatrix}
{\small s}\widetilde{E}{\small -}\widetilde{A}_{0} & {\small 0} &
{\small .}
& {\small 0} \\
{\small -}\widetilde{A}_{1} & {\small s}\widetilde{E}{\small -}\widetilde{A}%
_{0} & {\small .} & {\small 0} \\
{\small .} & {\small .} & {\small .} & {\small .} \\
{\small 0} & {\small .} & {\small -}\widetilde{A}_{1} & {\small s}\widetilde{%
E}{\small -}\widetilde{A}_{0}%
\end{pmatrix}%
\begin{pmatrix}
{\small 0} \\ 
{\small .} \\
{\small .} \\
{\small x}_{0}%
\end{pmatrix}%
=0$
\end{center}
This  implies that $x_{0}=0$.
\end{proof}
The next proposition is the central result of this section.

\begin{proposition}  \label{prop 3.2}
The system $(\widetilde{\Sigma }_{k})$ is left invertible if and
only if the
subsystem $\left( \widetilde{E},\widetilde{A}_{0},\text{$\widetilde{B},$}%
\widetilde{C}\right) $ is also invertible.
\end{proposition}
\begin{proof}
Let \ $X_{k}=\left( {\small x}_{0}....{\small x}_{k}\right) ^{T}$ , $%
Y_{k}=\left( {\small y}_{0}....{\small y}_{k}\right) ^{T}$ \ be two
elements such that $(s\widetilde{D}_{k}-\widetilde{F}_{k})X_{k}=$
$Y_{k}$. If the operator $(s\widetilde{E}-\widetilde{A}_{0})$ is
invertible, then it is not hard to show that

\begin{center}
$%
\begin{pmatrix}
{\small x}_{0} \\
{\small .} \\
{\small .} \\
{\small x}_{k}%
\end{pmatrix}%
{\small =}%
\begin{pmatrix}
{\small R}_{0}{\small (s)} & {\small 0} &  {\small .} & {\small 0%
} \\
{\small R}_{1}{\small (s)} & {\small R}_{0}{\small (s)} &
{\small .} & {\small .} \\
{\small .}&{\small .} & {\small .}& {\small .}  \\

{\small R}_{k}{\small (s)} & {\small .} &  {\small R}_{1}{\small %
(s)} & {\small R}_{0}{\small (s)}%
\end{pmatrix}%
\begin{pmatrix}
{\small y}_{0} \\
{\small .} \\
{\small .} \\
{\small y}_{k}%
\end{pmatrix}%
$
\end{center}

where{\small \ }$R_{i}(s)=[(s\widetilde{E}-\widetilde{A}_{0})^{-1}\widetilde{%
A}_{1}]^{i}(s\widetilde{E}-\widetilde{A}_{0})^{-1}$.

Furthermore the transfer function of the system $(\widetilde{\Sigma
}_{k})$ is given by:

\begin{center}
$\Phi _{k}\left( s\right) =%
\begin{pmatrix}
{\small T}_{0}{\small (s)} & {\small 0} & {\small .}  & {\small 0%
} \\
{\small T}_{1}{\small (s)} & {\small T}_{0}{\small (s)} &
{\small .} & {\small .} \\
{\small .}&{\small .} & {\small .}& {\small .} \\

{\small T}_{k}{\small (s)} & {\small .} &  {\small T}_{1}{\small %
(s)} & {\small T}_{0}{\small (s)}%
\end{pmatrix}%
$
\end{center}

where $T_{0}(s)=\widetilde{C}R_{0}(s)\widetilde{B}$ \ and \ $T_{i}(s)=%
\widetilde{C}R_{i}(s)\widetilde{B}$. According to the proof of Lemma
\ref{lem 3.1},  we conclude that the system $(\widetilde{\Sigma
}_{k})$ is left invertible if
and only if the subsystem $\left( \widetilde{E},\widetilde{A}_{0},\widetilde{B%
},\widetilde{C}\right) .$
\end{proof}

As an easy corollary of this proposition one has the following
result.

\begin{corollary}  \label{cor 3.3}
The system $(\widetilde{\Sigma }_{k})$ is left invertible if the
following condition hold:
 \begin{equation}
\forall s \in \mathbb{C}
 \ \ \ker
\begin{bmatrix}
 s\widetilde{E}-\widetilde{A}_{0}& \widetilde{B} \\
\widetilde{C}& 0
\end{bmatrix}
=\left\{ 0\right\}
\end{equation}
\end{corollary}
\begin{proof}
It immediately results from Proposition  \ref{prop 2.4} and
Proposition  \ref{prop 3.2}.
\end{proof}

\end{document}